\documentclass[11pt]{amsart}
\usepackage{amsmath, amssymb, amsthm, amsfonts, mathrsfs}

\newtheorem{theorem}{Theorem}[section]
\newtheorem{proposition}[theorem]{Proposition}
\newtheorem{corollary}[theorem]{Corollary}

\newtheorem{lemma}{Lemma}[section]

\theoremstyle{definition}
\newtheorem{definition}{Definition}[section]

\DeclareSymbolFont{AMSb}{U}{msb}{m}{n}
\DeclareMathSymbol{\N}{\mathbin}{AMSb}{"4E}
\DeclareMathSymbol{\Z}{\mathbin}{AMSb}{"5A}
\DeclareMathSymbol{\R}{\mathbin}{AMSb}{"52}
\DeclareMathSymbol{\Q}{\mathbin}{AMSb}{"51}
\DeclareMathSymbol{\I}{\mathbin}{AMSb}{"49}
\DeclareMathSymbol{\C}{\mathbin}{AMSb}{"43}

\newcommand{\ndiv}{\hspace{-4pt}\not|\hspace{2pt}}

\theoremstyle{definition}

\begin{document}
\title{Bounded Rank-One transformations}
\author{Su Gao}
\address{Department of Mathematics\\ University of North Texas\\ 1155 Union Circle \#311430\\  Denton, TX 76203\\ USA}
\email{sgao@unt.edu}
\author{Aaron Hill}
\address{Department of Mathematics\\ University of North Texas\\ 1155 Union Circle \#311430\\  Denton, TX 76203\\ USA}
\email{Aaron.Hill@unt.edu}
\date{\today}
\thanks{The first author acknowledges the US NSF grants DMS-0901853 and DMS-1201290 for the support of his research. The second author acknowledges the US NSF grant DMS-0943870 for the support of his research.}

\begin{abstract}
We define the notion of canonical boundedness among rank-one transformations and use it to characterize the class of all bounded rank-one transformations with trivial centralizer. We also explicitly characterize totally ergodic rank-one transformations with bounded cutting parameter. Together with a recent result of Ryzhikov our results provide a simple procedure for determining whether a bounded rank-one transformation has minimal self-joinings of all orders purely in terms of the cutting and spacer parameters for the transformation.
\end{abstract}
\maketitle\thispagestyle{empty}

\section{Introduction}
Rank-one transformations have been used as a source of examples and counterexamples for the study of important dynamical properties in ergodic theory. This paper contributes to the understanding of rank-one transformations with respect to properties such as total ergodicity, trivial centralizer, weak mixing, and MSJ (minimal self-joinings of all orders). Rank-one transformations with or without such properties have been constructed and investigated by various authors (cf. \cite{AbdalaouiNogueiradelaRue} \cite{Ageev} \cite{Bjork} \cite{Bourgain} \cite{Chacon2} \cite{delJunco} \cite{delJuncoRudolph} \cite{delJuncoRaheSwanson} \cite{Ferenczi} \cite{King88} \cite{KingThouvenot} \cite{Ryzhikov}).

Rank-one transformations are usually described by a sequence $(q_n : n \in \N)$ of integers greater than 1, called the {\em cutting parameter}, and a doubly indexed sequence $(a_{n, i}: n \in \N, 0 < i < q_n)$ of non-negative integers, called the {\em spacer parameter}.  If the cutting and spacer parameters are both bounded, we say the transformation being described is {\em bounded}.  In other words, a rank-one transformation is bounded if it admits a pair of cutting and spacer parameters that are both bounded.

It is important to note that different pairs of cutting and spacer parameters can be used to describe the same rank-one transformation.  In particular, any bounded rank-one transformation can also be described by cutting and spacer parameters that are unbounded.  It will be helpful to establish the following convention. When we speak of a bounded rank-one transformation, we always tacitly fix a pair of cutting and spacer parameters that are bounded. We will refer to these fixed parameters as {\em the} cutting and spacer parameters for the transformation. Our results will be independent of the particular bounded parameters chosen to describe a transformation.

Ryzhikov \cite{Ryzhikov} recently showed that a bounded rank-one transformation has MSJ if and only if it is totally ergodic and has trivial centralizer.  In this paper we give simple methods for determining whether a bounded rank-one transformation has trivial centralizer and is totally ergodic. By the result of Ryzhikov, we obtain a simple method for determining whether a bounded rank-one transformation has MSJ.

It turns out that the notion of {\it canonical} cutting and spacer parameters will play a crucial role. This was defined by the authors in \cite{GaoHill}.  In general, given a pair of cutting and spacer parameters, an associated pair of canonical cutting and spacer parameters can be calculated. The operation of finding the canonical cutting and spacer parameters is idempotent. Also, the canonical spacer parameter is just a rearrangement of the given spacer parameter, and therefore it is bounded if and only if the given spacer parameter is bounded. However, the boundedness of the canonical cutting parameter is not correlated with the boundedness of the given cutting parameter. 

We say that a bounded rank-one transformation is {\it canonically bounded} if for its given pair of cutting and spacer parameters, the corresponding canonical cutting and spacer parameters are bounded. Canonically bounded rank-one transformations are bounded. Our first main result is that the calculation of the canonical cutting and spacer parameters is precisely what is needed to determine if a bounded rank-one transformation has trivial centralizer.

\begin{theorem}
\label{thmintrotrivialcentralizer}
A bounded rank-one transformation has trivial centralizer iff it is canonically bounded.
\end{theorem}
This in particular implies that a rank-one transformation is canonically bounded iff for {\it any} given pair of bounded cutting and spacer parameters, the canonical cutting and spacer parameters are bounded.  In other words, for a bounded rank-one transformation, being canonically bounded does not depend on the choice of the given parameters.

This characterization of trivial centralizer can be restated in terms of the cutting and spacer parameters as follows.

\begin{theorem}
\label{thmintrotrivialcentralizer2}
Let $(q_n)$ and $(a_{n,i})$ be the cutting and spacer parameters for a rank-one transformation $T$ and suppose the cutting and spacer parameters are both bounded.  Then $T$ has trivial centralizer iff there exists $k \in \N$ such that for all $N \in \N$ there exist $n, m,i,j$, with $N \leq n,m < N+k$, $0<i<q_n$ and $0<j<q_m$, such that $a_{n,i} \neq a_{m,j}$.
\end{theorem}

Our next main result of this paper is a characterization of total ergodicity for all rank-one transformations with a bounded cutting parameter.  

\begin{theorem}
\label{thmintrototalergodicity}
Let $(q_n)$ and $(a_{n,i})$ be the cutting and spacer parameters for a rank-one transformation $T$ and suppose the cutting parameter is bounded.  Then $T$ is totally ergodic iff for every $d>1$ and $N \in \N$, there are $n \geq N$ and $0 < i < q_n$ such that $d$ does not divide $h_N + a_{n,i}$.
\end{theorem}

The $h_N$ in the theorem above refers to the height of the stage-$N$ tower in the cutting and stacking definition of $T$; it is defined recursively by $h_0 = 1$ and $h_{n+1} = q_n h_n + \sum_{0 < i < q_n} a_{n,i}$.

We remark that a sufficient condition for totally ergodicity was obtained by Creutz and Silva in Theorem 7 of \cite{CreutzSilva}.  Their condition also deals with the congruence classes of the entries in the spacer parameter. 

With Ryzhikov's result and the fact that MSJ implies weak mixing implies total ergodicity, Theorems \ref{thmintrotrivialcentralizer} and \ref{thmintrototalergodicity} yield the following corollary.

\begin{corollary}
Let $T$ be a canonically bounded rank-one transformation with the canonical cutting and spacer parameters $(q_n)$ and $(a_{n,i})$.  Then the following are equivalent:
\begin{enumerate}
\item  $T$ has minimal self-joinings of all orders.
\item  $T$ is weak mixing.
\item  $T$ is totally ergodic.
\item  For every $d>1$ and $N \in \N$, there are $n \geq N$ and $0 < i < q_n$ so that $d$ does not divide $h_N + a_{n,i}$.
\end{enumerate}
\end{corollary}

The rest of the paper is organized as follows. In Section 2 we review the constructive geometric (cutting and stacking) definition and the constructive symbolic definition of rank-one transformations, and discuss canonical generating sequences.  We give the preliminary facts that are needed for the proofs of our main results and show that Theorems \ref{thmintrotrivialcentralizer} and \ref{thmintrotrivialcentralizer2} are equivalent. In Section 3 we give the proof of Theorem \ref{thmintrotrivialcentralizer}. In Section 4 we prove Theorem \ref{thmintrototalergodicity}. In the last section we make some brief concluding remarks.

\section{Preliminaries}\label{sectionpreliminaries}

In this section we review the preliminaries about rank-one transformations. Unexplained definitions and proofs can be found in the authors' earlier work \cite{GaoHill}. We have tried to isolate all the needed definitions and results so that the rest of this paper is as self-contained as possible.

There are several standard definitions for rank-one transformations (cf. \cite{Ferenczi} for a more comprehensive list of these definitions).  In this paper we will use two most common definitions for rank-one transformations: the constructive geometric definition and the constructive symbolic definition.  

The following is the constructive geometric definition for a transformation on the unit interval $[0,1]$ to be rank-one. The definition describes a recursive {\em cutting and stacking} process that produces infinitely many Rokhlin towers to approximate the transformation.

\begin{definition}\label{defgeo} A measure preserving transformation $T$ on $[0,1]$ is {\em rank-one} if there exist sequences of positive integers $q_n>1$, $n \in\mathbb{N}$, and nonnegative integers $a_{n,i}$, $n \in \mathbb{N}$, $0< i< q_n$, such that, if $h_n$ is defined by
$$h_0 = 1, h_{n+1} = q_nh_n + \sum_{0<i<q_n} a_{n,i},$$
then
$$ \sum^{+\infty}_{n=0}\displaystyle\frac{h_{n+1}-q_n h_n}{h_{n+1}}<+\infty,$$
and subsets of $[0,1]$, denoted by $B_n$, $n\in\mathbb{N}$, by $B_{n,i}$, $n\in\mathbb{N}$, $0< i\leq q_n$, and by $C_{n,i,j}$, $n\in\mathbb{N}$, $0< i< q_n$, $0< j \leq a_{n,i}$, (if $a_{n,i} = 0$ then there are no $C_{n,i,j}$), such that for all $n$
\begin{itemize}
\item $\{B_{n,i}\mid 0< i\leq q_n\}$ is a partition of $B_n$,
\item the $T^k(B_n)$, $0\leq k< h_n$, are disjoint, 
\item $T^{h_n}(B_{n, i}) = C_{n,i,1}$ if $a_{n,i}\neq 0$ and $i < q_n$,
\item $T^{h_n}(B_{n,i}) = B_{n,i+1}$ if $a_{n,i} = 0$ and $i < q_n$,
\item $T(C_{n,i,j}) = C_{n,i,j+1}$ if $j < a_{n,i}$,
\item $T(C_{n,i,a_{n,i}}) = B_{n,i+1}$ if $i < q_n$,
\item $B_{n+1} = B_{n,1}$,
\end{itemize}
and the collection $\bigcup_{n=0}^\infty \{B_n, T(B_n), \dots, T^{h_n-1}(B_n)\}$ is dense in the measure algebra of all measurable subsets of $[0,1]$.
\end{definition}

In this definition the sequence $(q_n)$ is called the {\em cutting parameter}, the sets $C_{n,i,j}$ are called the {\em spacers}, and the doubly-indexed sequence $(a_{n,i})$ is called the {\em spacer parameter}. For each $n$, the collection $\{B_n, T(B_n),\dots, T^{h_n-1}(B_n)\}$ gives the {\it stage-$n$ tower}, with $B_n$ as the {\it base} of the tower, and each $T^k(B_n)$, where $0\leq k<h_n$, a {\it level} of the tower. The stage-$n$ tower has {\it height} $h_n$. At stage $n+1$, the stage-$n$ tower is cut into $q_n$ many {\it $n$-blocks} of equal measure. Each block has a base $B_{n,i}$ for some $0<i\leq q_n$ and has height $h_n$. These $n$-blocks are then stacked up, with spacers inserted in between. At future stages, these $n$-blocks are further cut into thinner blocks, but they always have height $h_n$. When we work with a rank-one transformation we will use both the terminology and the notation in this paragraph.

A general measure preserving transformation (on a general Lebesgue space) is {\em rank-one} if it is isomorphic to a constructive rank-one transformation on $[0,1]$ as defined above.  Two rank-one transformations are isomorphic if and only if they can be described by the same cutting and spacer parameters. However, different cutting and spacer parameters can describe the same rank-one transformation. 

We next turn to the constructive symbolic definition of rank-one systems. We will be talking about finite words over the alphabet $\{0,1\}$. For any finite word $\alpha$ we let $|\alpha|$ denote the length of $\alpha$. Let $\mathcal{F}$ be the set of all finite words over the alphabet $\{0,1\}$ that start and end with $0$. A {\em generating sequence} is an infinite sequence $(v_n)$ of finite words in $\mathcal{F}$ defined by induction on $n\in\mathbb{N}$:
$$ v_0=0,\ \ v_{n+1}=v_n1^{a_{n,1}}v_n1^{a_{n,2}}\dots v_n1^{a_{n,q_n-1}}v_n $$
for some positive integers $q_n> 2$ and nonnegative integers $a_{n,i}$ for $0< i< q_n$. We continue to refer to the sequence $(q_n)$ as the {\em cutting parameter} and the doubly-indexed sequence $(a_{n,i})$ as the {\em spacer parameter}. Note that the cutting and spacer parameters uniquely determine a generating sequence. 

A generating sequence converges to an infinite {\it rank-one word} $V\in\{0,1\}^{\mathbb N}$. We write $V=\lim_{n\to\infty} v_n$. Alternatively, an infinite word $V\in\{0,1\}^\mathbb{N}$ is a rank-one word if there is a generating sequence $(v_n)$ such that $V\upharpoonright |v_n|=v_n$ for all $n\in\mathbb{N}$. A pair of cutting and spacer parameters is said to be {\it trivial} if the infinite rank-one word induced is periodic; otherwise it is {\it nontrivial}. 

\begin{definition}
Given an infinite rank-one word $V$, the {\em (symbolic) rank-one (topological dynamical) system} induced by $V$ is a pair $(X,\sigma)$, where
$$X=X_V=\{ x\in\{0,1\}^\mathbb{Z}\,:\, \mbox{every finite subword of $x$ is a subword of $V$}\}$$
and $\sigma: X\to X$ is the {\em shift map} defined by
$$ \sigma(x)(k)=x(k+1) \mbox{ for all $k\in\mathbb{Z}$.} $$
\end{definition}

Trivial cutting and spacer parameters induce periodic rank-one words, and they in turn give rise to finite rank-one systems; such systems are said to be {\it degenerate}. On the other hand, if the cutting and spacer parameters are nontrivial, then the infinite rank-one word induced is aperiodic, and the induced rank-one system is a Cantor system (in particular uncountable); we call such systems {\it nondegenerate}. 

If $X$ is a rank one system, $\alpha$ is a finite word and $k\in\mathbb{Z}$, then
$$ U_{\alpha,k}=\{ x\in X\,:\, \mbox{$x$ has an occurrence of $\alpha$ (starting) at position $k$}\} $$
is a basic open set of $X$. For any nondegenerate rank one system $X$ there is an atomless shift-invariant (possibly infinite) measure $\mu_0$ on $X$ defined by
$$ \mu_0(U_{\alpha,k})=\lim_{n\to\infty}\displaystyle\frac{\mbox{the number of occurrences of $\alpha$ in $v_n$}}{\mbox{the number of occurrences of $0$ in $v_n$}}, $$
where $(v_n)$ is any generating sequence for $X$. It can be shown that $\mu_0$ depends only on $V$ and does not depend on the choice of the generating sequence. In fact,
$\mu_0$ is the unique shift-invariant measure on $X$ with $\mu_0(U_{0,0})=1$. If $\mu_0$ is finite (equivalently, $\mu_0(U_{1,0})$ is finite), then its normalization $\mu$ is given by
$$ \mu(U_{\alpha,k})=\lim_{n\to\infty}\displaystyle\frac{\mbox{the number of occurrences of $\alpha$ in $v_n$}}{|v_n|}, $$
where $(v_n)$ is any generating sequence for $X$. Again it can be shown that the definition depends only on $V$ and not on the particular generating sequence $(v_n)$. We denote by $\mathcal{R}^*$ the set of all aperiodic rank-one words $V$ for which the induced $\mu_0$ is finite (and therefore $\mu$ is defined).
For $V\in\mathcal{R}^*$, $\mu$ is the unique shift-invariant, atomless, probability Borel measure on $X$ (and is therefore ergodic). We summarize this in the following definition of symbolic rank one measure preserving systems.

\begin{definition}\label{rankonesymbolic} A {\em (symbolic) rank one (measure preserving) system} is a triple $(X,\mu,\sigma)$ such that $(X,\sigma)$ is a nondegenerate rank one topological dynamical system, and $\mu$ is the unique shift-invariant, atomless, probability Borel measure on $X$, provided that
$$ \lim_{n\to\infty}\displaystyle\frac{\mbox{the number of occurrences of $1$ in $v_n$}}{\mbox{the number of occurrences of $0$ in $v_n$}}<+\infty $$
for any generating sequence $(v_n)$ of $X$.
\end{definition}

In the rest of this paper we will work with both geometric rank-one transformations and symbolic rank-one systems. For this it is important to note the following basic fact: if the cutting and spacer parameters are nontrivial, then the two finiteness conditions in Definitions~\ref{defgeo} and \ref{rankonesymbolic} are equivalent, and the geometric rank-one transformation and the symbolic rank-one system given by the same parameters are isomorphic. When the parameters are trivial, the geometric rank-one transformation is an odometer map and the symbolic rank-one system is degenerate, and therefore they are not isomorphic. 

It has been proved in \cite{GaoHill} (Proposition 2.36) that for every nondegenerate rank-one system $X$ there is a unique infinite rank-one word $V$ (necessarily aperiodic) with $X=X_V$. Thus if two pairs of cutting and spacer parameters describe the same symbolic rank-one system, their corresponding generating sequences must converge to the same rank-one word. This is not true for geometric rank-one transformations. Any geometric rank-one transformation can be described by two pairs of cutting and spacer parameters corresponding to different rank-one words.

Canonical generating sequences and canonical cutting and spacer parameters are defined in the symbolic context. To define canonicity we need to recall more concepts and results from \cite{GaoHill}. 

Recall that $\mathcal{F}$ is the set of all finite words over the alphabet $\{0,1\}$ that start and end with $0$. If $u, v\in \mathcal{F}$, we say that $u$ is {\em built from} $v$, and denote $v\prec u$, if there is a positive integer $q>1$ and nonnegative integers $a_1,\dots, a_{q-1}$ such that
$$ u=v1^{a_1}v\dots v1^{a_{q-1}}v. $$
Moreover, we say that $u$ is {\em built simply from} $v$, and denote $v\prec_s u$, if $$a_1=\dots=a_{q-1}.$$
We write $u\preceq v$ if $u\prec v$ or $u=v$, and $u \preceq_s v$ if $u\prec_s v$ or $u=v$.
If $V$ is an infinite rank-one word and $v\in\mathcal{F}$, we say that $V$ is {\em built from} $v$ if
there are nonnegative integers $a_1,\dots, a_n, \dots$ such that
$$ V=v1^{a_1}v\dots v1^{a_n}v\dots. $$
With this notation, $(v_n)$ is a generating sequence iff $v_0=0$ and $v_n\prec v_{n+1}$ for all $n\in\mathbb{N}$. 

\begin{definition}\label{definition:canonical} An infinite generating sequence $(v_n)$ is {\em canonical} if it enumerates, in increasing order of length, all finite words $v$ from which $V=\lim_{n\to\infty} v_n$ is built, with the property that there are no $u, w\in\mathcal{F}$ such that
\begin{itemize}
\item[(i)] $V$ is built  from both $u$ and $w$,
\item[(ii)] $u\prec v\prec w$, and
\item[(iii)] $u\prec_s w$.
\end{itemize}
A pair of cutting and spacer parameters is {\em canonical} if the generating sequence determined by the parameters is canonical.
\end{definition}

One of the main results of \cite{GaoHill} is the following theorem.

\begin{theorem}[\cite{GaoHill} Proposition 2.15] Every aperiodic infinite rank-one word has a unique canonical generating sequence.
\end{theorem}

It follows that every nondegenerate symbolic rank-one system has a unique pair of canonical cutting and spacer parameters. For geometric rank-one transformations, this can be translated to the following. If a geometric rank-one transformation is described by a pair of nontrivial cutting and spacer parameters, then there is a unique pair of canonical cutting and spacer parameters for the given parameters. In general, geometric rank-one transformations do not correspond uniquely to infinite rank-one words, and therefore there might be more than one pair of canonical cutting and spacer parameters which can describe the same transformation.

In principle, Definition \ref{definition:canonical} has given an algorithm to calculate the canonical generating sequence given any infinite rank-one word. The algorithm produces an infinite generating sequence only when the infinite rank-one word is aperiodic, in which case the canonical generating sequence produced converge to the same rank-one word. It follows that, given any nontrivial cutting and spacer parameters, the canonical cutting and spacer parameters can also be calculated. 

The following are some key properties about the relations $\preceq$ and $\preceq_s$ and about the canonical generating sequence.  

For the following discussions fix an infinite rank-one word $V$ and consider the set $A_V$ of all finite words $v\in\mathcal{F}$ from which $V$ is built. For $u,v\in\mathcal{F}$ we say that $u$ and $v$ are {\it comparable} if either $u\prec v$, or $u=v$, or $v\prec u$; otherwise $u$ and $v$ are said to be {\it incomparable}. It was shown in \cite{GaoHill} Proposition 2.9 that $(A_V,\preceq)$ is a lattice. Thus given $u,v\in A_V$ it make sense to speak of $u\wedge v$, the greatest lower bound of $u$ and $v$, and $u\vee v$, the least upper bound of $u$ and $v$. We also have the following results.

\begin{lemma}[\cite{GaoHill} Proposition 2.13]\label{lemma:comparable} If $v\in\mathcal{F}$ is an element of the canonical generating sequence for $V$ and $u\in A_V$, then $u$ and $v$ are comparable.
\end{lemma}

\begin{lemma}[\cite{GaoHill} Lemma 2.7]\label{lemma:simplybuilt} {\rm (a)} Suppose $u, v, u', v'\in \mathcal{F}$ and $u\prec v\prec u'\prec v'$. If $u\prec_s u'$ and $v\prec_s v'$, then $u\prec_s v'$.

{\rm (b)} Suppose $u, v, w\in\mathcal{F}$ and $u\prec v\prec w$. If $u\prec_s w$, then $u\prec_s v$ and $v\prec_s w$.
\end{lemma}

\begin{lemma}[\cite{GaoHill} Lemma 2.10]\label{lemma:simplybuilt2} If $u, v\in A_V$, then $(u\wedge v)\preceq u \preceq (u\vee v)$, $(u\wedge v)\preceq v \preceq (u\vee v)$, and $(u\wedge v)\preceq_s (u\vee v)$. 
\end{lemma}

We prove the following lemmas to illustrate how to work with these notions.

\begin{lemma}\label{lemma:left}
Let $u, v, w\in A_V$ be such that $u\prec v\prec w$ and $u\prec_s w$. Suppose $u$ is the shortest word  with the properties $u\prec v\prec w$ and $u\prec_s w$. Then $u$ is an element of the canonical generating sequence.
\end{lemma}
\begin{proof}
Assume $u$ is not an element of the canonical generating sequence. By Definition~\ref{definition:canonical} there are $u', w'\in A_V$ such that $u'\prec u\prec w'$ and $u'\prec_s w'$. 

If $w$ and $w'$ are comparable, we have either $w\preceq w'$ or $w'\prec w$. If $w\preceq w'$ then $u'\prec u\prec v\prec w\preceq w'$, and by Lemma \ref{lemma:simplybuilt} (b) we have $u'\prec_s w$, which contradicts the minimality assumption on the length of $u$. On the other hand, if $w'\prec w$, then we have $u'\prec u\prec w'\prec w$, where $u'\prec_s w'$ and $u\prec_s w$. In this case, by Lemma~\ref{lemma:simplybuilt} (a) we have $u'\prec_s w$, and thus $u'\prec u\prec v\prec w$ with $u'\prec_s w$, again contradicting the minimality assumption on the length of $u$.

We next assume that $w$ and $w'$ are not comparable. Then by Lemma~\ref{lemma:simplybuilt2} we have $u'\prec u\preceq (w\wedge w')\prec w'\prec (w\vee w')$ and  $(w\wedge w')\prec_s (w\vee w')$. Since $u'\prec_s w'$, by Lemma~\ref{lemma:simplybuilt} (a) we have $u'\prec_s (w\vee w')$. Thus we have $u'\prec u\prec v\prec w\prec (w\vee w')$ with $u'\prec_s (w\vee w')$. By Lemma \ref{lemma:simplybuilt} (b) we have $u'\prec_s w$, again contradicting the minimality of the length of $u$.

The proof of the lemma is complete.
\end{proof}

The following lemma is similar to Lemma \ref{lemma:left} and has a similar proof.

\begin{lemma}\label{lemma:right} Let $u, v, w\in A_V$ be such that $u\prec v\prec w$ and $u\prec_s w$. Suppose $w$ is the longest word with the properties $u\prec v\prec w$ and $u\prec_s w$. Then $w$ is an element of the canonical generating sequence.
\end{lemma}

We will need the following lemma in the next section.

\begin{lemma}\label{lemma:main} Let $(v_n)$ be a canonical generating sequence. Suppose $v\in\mathcal{F}$ and for some $n\in\mathbb{N}$,
$v_n\prec v\prec v_{n+1}$. Then $v_n\prec_s v_{n+1}$.
\end{lemma}

\begin{proof} Since $v$ is not an element of the canonical generating sequence, there are $u, w\in A_V$ with $u\prec v\prec w$ and $u\prec_s w$. First fix an arbitrary such $w$. Take $u_0$ to be the shortest word with the properties with $u_0\prec v\prec w$ and $u_0\prec_s w$. By Lemma~\ref{lemma:left} $u_0$ is an element of the canonical generating sequence. Thus $u_0\preceq v_n$.

Now if there exists a longest word $w_0\in A_V$ with the properties $u_0\prec v\prec w_0$ and $u_0\prec_s w_0$, then by Lemma~\ref{lemma:right} this $w_0$ must be an element of the canonical generating sequence, and hence $v_{n+1}\preceq w_0$. By Lemma \ref{lemma:simplybuilt} (b), $v_n\prec_s v_{n+1}$.

If there are arbitrarily long words $w_0\in A_V$ with the properties $u_0\prec v\prec w_0$ and $u_0\prec_s w_0$, we can take $w_0$ to be an arbitrary such word with $|w_0|\geq |v_{n+1}|$. By Lemma \ref{lemma:comparable} $v_{n+1}\preceq w_0$. Now $u_0\preceq v_n\prec v_{n+1}\preceq w_0$ with $u_0\prec_s w_0$. By Lemma \ref{lemma:simplybuilt} (b) again, $v_n\prec_s v_{n+1}$.
\end{proof}


The following proposition shows that Theorems \ref{thmintrotrivialcentralizer} and \ref{thmintrotrivialcentralizer2} are equivalent.

\begin{proposition}\label{propequivalent}
Let $(q_m)$ and $(a_{m,i})$ be the cutting and spacer parameters for a rank-one transformation and suppose $(q_m)$ is bounded. Then the canonical cutting parameter is unbounded iff for all $k\in \N$ there exists $N\in\N$ such that for all $m,m^\prime, i,i^\prime $ with $N\leq m, m^\prime <N+k$, $0<i<q_m$ and $0<i^\prime<q_{m^\prime}$, we have $a_{m,i}=a_{m^\prime,i^\prime}$.
\end{proposition}
\begin{proof}
Let $(w_m)$ be the generating sequence corresponding to $(q_m)$ and $(a_{m,i})$.  Let $(v_n)$ be the corresponding canonical generating sequence.

Suppose for for all $m,m^\prime, i,i^\prime $ with $N\leq m, m^\prime <N+k$, $0<i<q_m$ and $0<i^\prime<q_{m^\prime}$, we have $a_{m,i}=a_{m^\prime,i^\prime}$.  Fix $k \in \N$.  The assumption implies that there exists $N \in \N$ such that $w_N \prec_s w_{N+k}$.  
By Lemma \ref{lemma:comparable} each $v_n$ is comparable to each $w_m$.  Let $n$ be as large as possible such that $v_n \preceq w_N$.  We must have $w_{N+k} \preceq v_{n+1}$, by the definition of the canonical generating sequence.  Then the canonical cutting parameter at position $n$ is at least $q_N q_{N+1} \ldots  q_{N + k -1} \geq  2^k$.  Since $k$ was arbitrary, the canonical cutting parameter is unbounded.

Conversely suppose the canonical cutting parameter is unbounded.  Let $Q$ be an upper bound for $(q_m)$.  For a finite word $\alpha$, let $Z(\alpha)$ denote the number of 0s in $\alpha$.  Note that $Z(w_{m + 1}) = q_{m} Z(  w_{m}) \leq Q Z(  w_{m})$.  Fix $k \in \N$ and choose $n$ so that $$\frac{Z(v_{n+1})}{Z(v_n)} \geq Q^{k+1}.$$  Let $N$ be as small as possible such that $v_n \preceq w_N$.  It follows that $$v_n \preceq w_N \prec w_{N+1} \prec \ldots \prec w_{N + k} \preceq v_{n+1}.$$  By Lemma~\ref{lemma:comparable}, we have $v_n \prec_s v_{n+1}$.  By Lemma \ref{lemma:simplybuilt} (b), $w_{N } \prec_s w_{N+k}$.  Thus, for all $m,m^\prime, i,i^\prime $ with $N\leq m, m^\prime <N+k$, $0<i<q_m$ and $0<i^\prime<q_{m^\prime}$, we have $a_{m,i}=a_{m^\prime,i^\prime}$.

\end{proof}

\section{Characterizing Trivial Centralizer}

In this section we work with bounded rank-one transformations and characterize those with trivial centerlizer. We first define our terminology. 

\begin{definition} (a) A rank-one transformation (or a symbolic rank-one system) is {\em bounded} if it can be described by a pair of cutting and spacer parameters that are both bounded.

(b) A rank-one transformation (or a symbolic rank-one system) is {\em canonically bounded} if it can be described by a pair of canonical cutting and spacer parameters that are both bounded.
\end{definition}

Note that odometer maps can always be described by bounded (albeit trivial) parameters, and therefore are bounded rank-one transformations. It will follow from our results below that they are not canonically bounded. Since canonical parameters are always nontrivial, canonically bounded symbolic rank-one systems are nondegenerate. 

A well-known example of bounded rank-one transformation is Chacon's tranformation. A generating sequence is given by
$$ v_0=0,\  v_{n+1}=v_nv_n1v_n. $$
This is canonical because there are no other finite words from which the corresponding infinite rank-one word is built and each $v_{n+2}$ is not built simply from $v_n$. It follows that Chacon's transformation is canonically bounded.

The following example was considered by del Junco and Rudolph in \cite{delJuncoRudolph}. Let $\Delta$ denote the set of all triangular numbers, i.e., the set of all $n(n+1)/2$ for positive integers $n$. Consider the generating sequence given by
$$ v_0=0;\ v_n=v_{n-1}1v_{n-1} \mbox{ if $n\in\Delta$;}\  v_n=v_{n-1}v_{n-1} \mbox{ otherwise.} $$
The resulting rank-one system is bounded. The corresponding canonical generating sequence, however, is given by
$$ u_0=0; \ u_n=u_{n-1}1u_{n-1} \mbox{ if $n$ is odd};\ u_n=u_{n-1}^{2^n} \mbox{ otherwise}. $$
Thus the rank-one system is not canonically bounded.

Our objective of this section is to characterize the property of having trivial centralizer for bounded rank-one transformations. For a Lebesgue space $(X,\mu)$ let $\textnormal{Aut}(X, \mu)$ denote the group of all invertible measure preserving transformations, where two transformations are identified if they differ only on a null set. With the weak topology, $\textnormal{Aut}(X, \mu)$ is a Polish group. If $T\in \textnormal{Aut}(X, \mu)$, the {\it centralizer} (or {\it commutant}) of $T$ is
$$ C(T)=\{ S\in \textnormal{Aut}(X, \mu)\,:\, TS=ST\}. $$
The classical theorem of King \cite{King1} states that if $T$ is rank-one then $C(T)$ is the weak closure of the set of integral powers of $T$ in $\textnormal{Aut}(X, \mu)$. 

An invertible measure preserving transformation $T$ is said to have {\it trivial centralizer} if $C(T)=\{T^n\,:\, n\in\mathbb{Z}\}$. For a rank-one transformation $T$, this is equivalent to saying that the set of integral powers of $T$ is a discrete subgroup of $\textnormal{Aut}(X, \mu)$. In particular, if there is an increasing sequence $(n_k)$ of positive integers such that $T^{n_k}$ converges to the identity map in the weak topology, then $C(T)$ is perfect, and therefore $T$ does not have trivial centralizer.

Del Junco has shown in \cite{delJunco} that Chacon's transformation has trivial centralizer.

We are now ready to prove Theorem \ref{thmintrotrivialcentralizer}, i.e., a bounded rank-one transformation has trivial centralizer iff it is canonically bounded.  We separate the two directions of argument into the following theorems.

\begin{theorem}
\label{thmnottrivialcentralizer}
If a bounded nondegenerate symbolic rank-one system is not canonically bounded, then it does not have trivial centralizer.
\end{theorem}

\begin{theorem}
\label{thmtrivialcentralizer}
If a symbolic rank-one system is canonically bounded, then it has trivial centralizer.
\end{theorem}

Since Theorem~\ref{thmintrotrivialcentralizer} was stated for geometric rank-one transformations, we need to see that Theorems~\ref{thmnottrivialcentralizer} and \ref{thmtrivialcentralizer} are sufficient to establish Theorem \ref{thmintrotrivialcentralizer}. For this let $T$ be a bounded rank-one transformation and fix a pair of cutting and spacer parameters $(q_n)$ and $(a_{n,i})$ that are both bounded. If the pair of parameters is nontrivial, then it also describes a nondegenerate symbolic rank-one system that is isomorphic to $T$. In this case Theorems~\ref{thmnottrivialcentralizer} and \ref{thmtrivialcentralizer} apply. On the other hand, if the pair of parameters is trivial, then $T$ is an odometer map and does not have trivial centralizer. In this case it suffices to argue that $T$ is not canonically bounded. Suppose it were, then $T$ could be described by a pair of canonical cutting and spacer parameters that are both bounded. In particular the pair of parameters is nontrivial, and hence it also describes a nondegerate symbolic rank-one system that is isomorphic to $T$. By Theorem \ref{thmtrivialcentralizer} the symbolic rank-one system has trivial centralizer, and therefore so does $T$, a contradiction.

\subsection{Nontrivial centralizer}

In this subsection we prove Theorem~\ref{thmnottrivialcentralizer}. For this we in fact prove the following stronger result without assuming the boundedness of the spacer parameter.

\begin{proposition}
\label{propnottrivialcentralizer}
Let $(X, \mu, \sigma)$ be a symbolic rank-one system with canonical cutting and spacer parameters $(q_n)$ and $(a_{n,i})$, and suppose that $(q_n)$ is unbounded.  Suppose that $(X, \mu, \sigma)$ can also be described by the cutting and spacer parameters $(q^\prime_m)$ and $(a_{m,j}^\prime)$ with $(q^\prime_m)$ bounded. Then $(X, \mu, \sigma)$ does not have trivial centralizer.  
\end{proposition}

\begin{proof}
Let $Q$ be a bound for $(q^\prime_m)$.  We show that $(X, \mu, \sigma)$ does not have trivial centralizer in two steps.  First we will find an increasing sequence $(n_k)$ of positive integers so that 
\begin{enumerate}
\item  $q_{n_k} \rightarrow \infty$, and
\item  for all $k \in \N$ and $0 < i,i^\prime < q_{n_k}$, $a_{n_k, i } = a_{n_k, i^\prime}$.
\end{enumerate}
Then, with such a sequence, we define $r_k = h_{n_k} + a_{n_k, 1}$ and show that $\sigma^{r_k}$ converges to the identity map in the weak topology.

Since the canonical cutting parameter $(q_n)$ is not bounded, we can choose an increasing sequence $(n_k)$ of positive integers so that $$Q < q_{n_0} < q_{n_1} < q_{n_2} < \ldots.$$  
We claim that for all $k \in \N$ and  $0 < i,i^\prime < q_{n_k}$, $a_{n_k, i } = a_{n_k, i^\prime}.$

To prove this, let $V$ denote the rank-one word corresponding the symbolic system $(X, \mu, \sigma)$.  Let $(v^\prime_m)$ denote the generating sequence corresponding to the parameters $(q^\prime_m)$ and $(a^\prime_{m,j})$.  Let $(v_n)$ denote the canonical generating sequence; $(v_n)$ corresponds to the canonical parameters $(q_n)$ and $(a_{n,i})$.  Fix $k \in \N$.

We argue that there is some $m \in \N$ so that $v_{n_k} \prec v'_m \prec v_{n_k+1}$. This is similar to some arguments in the proof of Proposition~\ref{propequivalent}. In fact, we use the following notation from that proof. For a finite word $\alpha$, let $Z(\alpha)$ denote the number of 0s in $\alpha$.  Note that $Z(v_{n_k + 1}) = q_{n_k} Z(  v_{n_k})$.  But for each $m \in \N$, $Z(v^\prime_{m + 1}) = q^\prime_m Z(v^\prime_m)$.  Since $q_{n_k} > Q$, there must be some $m \in \N$ so that $$Z(v_{n_k}) < Z (v^\prime_m) < Z(v_{n_k+1}).$$  Since $v_{n_k}$ and $v_{n_k + 1}$ are both in the canonical generating sequence for $V$, they are comparable with every finite word from which $V$ is built, by Lemma~\ref{lemma:comparable}.  It follows that $v_{n_k} \prec v'_m \prec v_{n_k + 1}$.

Now by Lemma \ref{lemma:main} we know that $v_{n_k} \prec_s v_{n_k + 1}$, which implies that for all $0<i,i^\prime < q_{n_k}$, $a_{n_k, i } = a_{n_k, i^\prime}$. This proves our claim.

We now have an increasing sequence $n_k$ satisfying (1) and (2) above. Let $r_k = h_k + a_{n_k, 1}$.  We need to show that $\sigma^{r_k}$ converges to the identity map in the weak topology.  To do so, let $T$ be a geometric rank-one transformation that is described by $(q_n)$ and $(a_{n,i})$. Let $\lambda$ be the Lebesgue measure on $[0,1]$. Since the two measure preserving systems $([0,1],\lambda,T)$ and $(X,\mu,\sigma)$ are isomorphic, it suffices to show that $T^{r_k}$ converges weakly to the identity map on $[0,1]$.

We need to show that for every Lebesgue measurable set $A\subseteq [0,1]$, $$\lambda [ T^{r_k} (A) \triangle A ] \rightarrow 0.$$  Since the levels of the Rokhlin towers are dense in the measure algebra, it suffices to show the above when $A$ is the level of a tower, or in fact, when $A$ is the base of one of the towers. 

Let $B_n$ denote the base of the stage-$n$ tower in the construction of $T$.  Fix $N \in \N$.  We claim that $$\lambda[ T^{r_k} (B_N) \triangle B_N] \leq \frac{2\mu[B_N]}{q_{n_k}}.$$
Since $q_{n_k} \rightarrow \infty$, this would imply $\lambda[ T^{r_k} (B_N) \triangle B_N] \rightarrow 0$ as needed.  

Define  $$I _k= \{0 \leq i < h_{n_k} : T^i (B_{n_k}) \subseteq B_N\}$$ and $$J_k =  \{0 \leq j < h_{n_k+1} : T^j (B_{n_k +1}) \subseteq B_{n_k}\}.$$  We know that for each $0 < i < q_{n_k}$, $a_{n_k, i} = r_k - h_{n_k}$.   Thus, $$J_k = \{0, r_k, 2r_k, \ldots, (q_{n} -1)r_k\}.$$  It is clear that $$B_N = \bigcup_{i \in I_k, j \in J_k}  T^{i + j} (B_{n_k+1})$$ and that the union above is disjoint.  Let $J_k^\prime = J_k \setminus \{(q_n -1)r_k\}$.  Let $$C =  \bigcup_{i \in I_k, j \in J_k^\prime}  T^{i + j} (B_{n_k+1}).$$ Notice that $C \subseteq B_N$ and $\lambda [C] = \frac{q_{n_k} - 1}{q_{n_k} } \lambda[B_N]$.  Notice also that $T^{r_k} (C) \subseteq B_N$.  These imply that $$\lambda[ T^{r_k} (B_N) \triangle B_N] \leq \frac{2\mu[B_N]}{q_{n_k}}.$$
\end{proof}

\subsection{Occurrences of finite words}
The rest of this section, consisting of two subsections, is devoted to a proof of Theorem \ref{thmtrivialcentralizer}. In this part we will work primarily in the symbolic setting. In the current subsection we develop some more tools for the study of symbolic rank-one systems.

Let $(X, \mu, \sigma)$ be a symbolic rank-one system. In the proofs of our results we will talk about {\it occurrences} of finite words (especially those words on the canonical generating sequence) in an element $x\in X$. It was shown in \cite{GaoHill} (Proposition 2.31) that, if $w$ is on the canonical generating sequence, then any $x\in X$ can be uniquely written as
$$ x=\dots w1^{a_{-1}}w1^{a_0}w1^{a_1}\dots $$
for nonnegative integers $a_i$, $i\in\Z$. We say that $x$ is {\it built from} $w$ and that the demonstrated occurrences of $w$ in $x$ are {\it expected}. It is possible for $x$ to contain unexpected occurrences of $w$, but $x$ will not be built from $w$ with unexpected occurrences. Two occurrences of $w$ in $x$ are {\it consecutive} if they are disjoint and the word occurring in between contains no $0$. In other words, when a word of the form $w1^aw$ occurs in $x$, the two demonstrated occurrences of $w$ are said to be consecutive, and the first occurrence is said to {\it precede} the second occurrence, and the second is said to {\it follow} the first one. A {\it string of occurrences} of $w$ in $x$ is nothing but an occurrence of a single finite word $u$ in $x$ where $u$ is built from $w$.

Now suppose $(W_n)$ is the canonical generating sequence for $X$. For each $n\in\N$ and $i\in \Z$ define
$$ E_{n,i}=\{x\in X\,:\, \mbox{ $x$ has an expected occurrence of $W_n$ beginning at $i$}\}. $$
We will need the following lemmas in subsequent proofs.

\begin{lemma}
\label{lemma:Enidensity}
For any $l\in \N$ the collection $\{E_{n,i}: n\geq l \textnormal{ and } i \in \Z\}$ is dense in the measure algebra of all measurable subsets of $X$.
\end{lemma}

\begin{proof}
The canonical generating sequence $(W_n)$ gives rise to a pair of canonical cutting and spacer parameters, which in turn describes a cutting and stacking rank-one transformation $T$ on $[0,1]$. The measure preserving systems $([0,1],\lambda, T)$ and $(X,\mu,\sigma)$ are isomorphic. Under this isomorphism, each $E_{n,0}$ corresponds exactly to the set $B_n$, the base of the stage-$n$ Rokhlin tower in the construction of $T$. In general, $E_{n,i}=\sigma^{-i}(E_{n,0})$ corresponds to $T^{-i}(B_n)$. By Definition~\ref{defgeo}, $\bigcup_{n\geq l} \{B_n, T(B_n),\dots, T^{h_n-1}(B_n)\}$ is dense in the measure algebra of all measurable subsets of $[0,1]$. This implies that $\{ E_{n,i}\,:\, n\geq l, i\in\mathbb{Z}\}$ is dense in the measure algebra of all measurable subsets of $X$.
\end{proof}

\begin{lemma}
\label{propW_{n+2}}
If $x \in X$ has an occurrence of $W_{n+2}$ beginning at $i$, then $x$ has an expected occurrence of $W_n$ beginning at $i$.
\end{lemma}

\begin{proof}
Assume that $x$ has an occurrence of $W_{n+2}$ beginning at $i$ and that the occurrence of $W_n$ beginning at $i$ is unexpected. Since $W_{n+2}$ is built from $W_n$, we assume 
$$ W_{n+2}=W_n1^{b_1}W_n1^{b_2}\dots 1^{b_{s-1}}W_n. $$
Since $x$ is built from $W_n$, there is a unique expected occurrence of $W_n$ in $x$ beginning at an $j$ with $i<j<i+|W_n|$. By keeping track of the number of 0s in $W_n$ we know that the next expected occurrence of $W_n$ in $x$ begins at $|W_n|+b_1+j$. Similarly, the following expected occurrence of $W_n$ in $x$ begins at $2|W_n|+b_1+b_2+j$, etc. It follows that $b_1=b_2=\dots=b_{s-1}$. This means that $W_n\prec_s W_{n+2}$. Since $W_n\prec W_{n+1}\prec W_{n+2}$, this is a contradiction to the assumption that $W_{n+1}$ is an element of the canonical generating sequence.
\end{proof}

\subsection{Trivial centralizer}
We give the proof of Theorem \ref{thmtrivialcentralizer} in this subsection. 

Let $(X, \mu, \sigma)$ be a symbolic rank-one system with canonical generating sequence $(W_n)$ and corresponding cutting and spacer parameters $(q_n)$ and $(a_{n,i})$.  Suppose both parameters are bounded and let $q_{\textnormal{max}}$ and $a_{\textnormal{max}}$ denote their respective maxima.  Fix $Q\geq q_{\textnormal{max}}, a_{\textnormal{max}}$. Let $\psi \in \textnormal{Aut}(X, \mu)$ commute with $\sigma$.  To prove the theorem we need to show that $\psi$ is an integral power of $\sigma$.

Fix $k \in \N$ so that $W_k$ contains $a_{\textnormal{max}}$ many consecutive 1s.  We first state a simple fact; its verification is left to the reader.

\begin{lemma}
\label{facts=t}
If $x \in X$ has occurrences of $W_k$ beginning at $i$ and $j$, then either $i = j$, or $|i-j| > a_{\textnormal{max}}$.
\end{lemma}

By Lemma~\ref{lemma:Enidensity}, we can find $l \geq k$ and $z \in \Z$ so that $$\frac{\mu [E_{l,z} \setminus \psi^{-1} (E_{k,0})]}{\mu [E_{l,z}]}  < \frac{1}{2Q^3 + 1}.$$
Fix such $l\geq k$ and $z\in \Z$.

Let $\phi = \psi \circ \sigma^{-z}$.  Note that $\phi \in \textnormal{Aut} (X, \mu)$ commutes with $\sigma$ and that $\phi$ is a power of $\sigma$ iff $\psi$ is a power of $\sigma$.    Thus, to prove the theorem it suffices to show that $\phi$ is an integral power of $\sigma$.  

It follows from the invariance of $\mu$ under $\sigma$ that $$\frac{\mu [E_{l,0} \setminus \phi^{-1} (E_{k,0})]}{\mu [E_{l,0}]} < \frac{1}{2Q^3 + 1}.$$

Note that $x \in E_{l,0} \cap \phi^{-1} (E_{k,0})$ iff $x$ has an expected occurrence of $W_l$ beginning at 0 and $\phi (x)$ has an expected occurrence of $W_k$ beginning at 0.  More generally, since $\phi$ commutes with $\sigma$, we have the following for all $i \in \Z$:  $\sigma^i (x) \in E_{l,0} \setminus \phi^{-1} (E_{k,0})$ iff $x$ has an expected occurrence of $W_l$ beginning at $i$ and $\phi (x)$ has an expected occurrence of $W_k$ beginning at $i$.  This prompts the following definition.

\begin{definition}  An expected occurrence of $W_l$ beginning at $i$ in $x \in X$ is {\em good} if $\phi(x)$ has an expected occurrence of $W_k$ beginning at $i$.  Otherwise it is {\em bad}.
\end{definition}

\begin{proposition}
\label{propmuae}
For $\mu$ almost every $x \in X$, exactly one of the following holds.
\begin{enumerate}
\item  Every expected occurrence of $W_l$ is good.
\item  For some $d>0$, there is a string of $d$ many bad occurrences of $W_l$ that is both preceded by a string of $2Q^3 d$ good occurrences of $W_l$ and followed by a good occurrence of $W_l$.
\end{enumerate}
\end{proposition}

\begin{proof}
By the ergodic theorem, for $\mu$ almost every $x \in X$, both of the following hold:

\begin{enumerate}
\item[(a)]  $\displaystyle \lim_{N \rightarrow \infty} \frac{|\{i \in [0,N] : \textnormal{ $x$ has a bad occ. of $W_l$ at $i$}\}|}{|\{i \in [0,N] : \textnormal{ $x$ has an exp. occ. of $W_l$ at $i$}\}|} <  \frac{1}{2Q^3 + 1}$

\item[(b)]  $\displaystyle \lim_{N \rightarrow -\infty} \frac{|\{i \in [N,0] : \textnormal{ $x$ has a bad occ. of $W_l$ at $i$}\}|}{|\{i \in [0,N] : \textnormal{ $x$ has an exp. occ. of $W_l$ at $i$}\}|} < \frac{1}{2Q^3 + 1}$
\end{enumerate}

For such an $x \in X$, every bad occurrence of $W_l$ is contained in a maximal string of bad occurrences of $W_l$ (one that is immediately preceded by a good occurrence of $W_l$ and immediately followed by a good occurrence of $W_l$).  If $x$ does have a maximal string of good occurrences of $W_n$, then (b) guarantees that some maximal string of bad occurrences of $W_l$ (say the string consists of $d$ bad occurrences of $W_l$) is immediately preceded by a string of $2Q^3d$ good occurrences of $W_l$. 
\end{proof}

In the analysis that follows we will show that condition (1) in the statement of the previous proposition implies that $\phi (x)$ is a shift of $x$ (see Proposition \ref{propallgood}).  We will also show that condition (2) in the statement of the previous proposition never happens (see Proposition \ref{propcontradiction}).  It will then follow easily from the ergodicity of $\sigma$ that $\phi$ is an integral power of $\sigma$.

There is a notion that will be repeatedly used in this analysis and that we will now introduce: $r_{x, i}$.  If $x \in X$ has a good occurrence of $W_l$ beginning at $i$, then $\phi (x)$ has an expected occurrence of $W_k$ beginning at $i$.  This expected occurrence of $W_k$ is completely contained in an expected occurrence of $W_l$.  Thus, there is a unique $r_{x, i}$ satisfying $0 \leq r_{x,i} \leq |W_l| - |W_k|$ such that $\phi (x)$ has an expected occurrence of $W_l$ beginning at $i - r_{x,i}$.

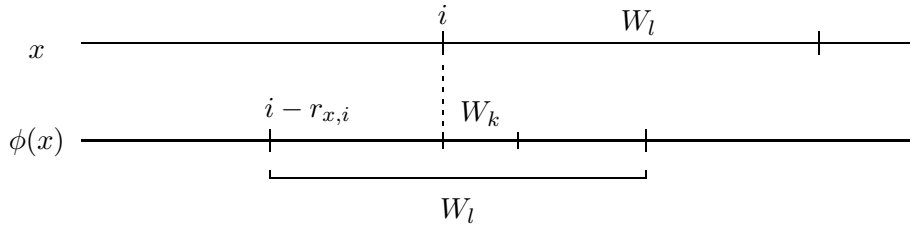
\begin{figure}[h]
\begin{center}
\setlength{\unitlength}{.1cm}
\begin{picture}(120,28)(0,0)

\put(0,-8){
\put(10,17){\line(1,0){110}}
\put(10,30){\line(1,0){110}}
\put(4,15){\makebox(0,0)[b]{$\phi(x)$}}
\put(4,28){\makebox(0,0)[b]{$x$}}
\put(40,19){\makebox(0,0)[b]{$i-r_{x,i}$}}
\put(63,19){\makebox(0,0)[b]{$W_k$}}
\put(84,31){\makebox(0,0)[b]{$W_l$}}
\put(60,6){\makebox(0,0)[b]{$W_l$}}
\put(58,32.5){\makebox(0,0)[b]{$i$}}
\put(35,12){\line(1,0){50}}
\put(35,12){\line(0,1){1}}
\put(85,12){\line(0,1){1}}
\put(35,15.5){\line(0,1){3}}
\put(58,16){\line(0,1){2}}
\put(68,16){\line(0,1){2}}
\put(58,28.5){\line(0,1){3}}
\put(108,28.5){\line(0,1){3}}
\put(85,15.5){\line(0,1){3}}
\multiput(58,27)(0,-1.5){6}{\line(0,-1){0.5}}
}
\end{picture}
\end{center}
\end{figure}

\begin{lemma}
\label{lemmabasic1}
Suppose $x \in X$ has an occurrence of $W_l 1^s W_l$ beginning at $i$ and the $W_l$ beginning at $i$ in $x$ is good 
(hence, $\phi(x)$ has an expected occurrence of $W_l$ beginning at $i - r_{x,i}$).  
Let $t \in \N$ be such that $\phi (x)$ has an occurrence of $W_l 1^t W_l$ beginning at $i- r_{x,i}$.  The following are equivalent:
\begin{enumerate}
\item $s = t$;
\item the occurrence of $W_l$ beginning at $i + |W_l| + s$ in $x$ is good.
\end{enumerate}
\end{lemma}

\begin{proof}
We know that $\phi (x)$ has an expected occurrence of $W_k$ that begins at $i$ and is completely contained in the expected occurrence of $W_l$ beginning at $i - r_{x,i}$.  The expected occurrence of $W_l$ beginning at $i - r_{x,i} + |W_l| + t$ in $\phi(x)$ must have an expected occurrence of $W_k$ at the same relative location, i.e., beginning at $i + |W_l| + t$.

If $s=t$, then $\phi (x)$ has an expected occurrence of $W_k$ beginning at $i + |W_l| + s$ and thus, the expected occurrence of $W_l$ beginning at $i + |W_l| + s$ in $x$ is good.

On the other hand, suppose the occurrence of $W_l$ beginning at $i + |W_l| + s$ in $x$ is good.  Then $\phi(x)$ has expected occurrences of $W_k$ beginning at $i + |W_l| + s$ and at $i + |W_l| + t$.  By Lemma \ref{facts=t}, either $s=t$ or $|s-t| >a_{\textnormal{max}}$.  Since $0 \leq s,t \leq a_{\textnormal{max}}$, it cannot be the case that $|s-t| >a_{\textnormal{max}}$.  Thus, $s=t$.
\end{proof}

What follows is a nearly identical lemma; we leave its verification to the reader.

\begin{lemma}
\label{lemmabasic2}
Suppose $x \in X$ has an occurrence of $W_l 1^s W_l$ beginning at $j - |W_l| - s$ and the $W_l$ beginning at $j$ in $x$ is good (hence, $\phi (x)$ as an expected occurrence of $W_l$ beginning at $j - r_{x,j}$).  Let $t \in \N$ be such that $\phi (x)$ has an occurrence of $W_l 1^t W_l$ beginning at $j - r_{x,j} - |W_l| - t$.  Then the following are equivalent.
\begin{enumerate}
\item $s = t$;
\item the occurrence of $W_l$ beginning at $j - |W_l| - s$ in $x$ is good.
\end{enumerate}
\end{lemma}

\begin{proposition}
\label{propallgood}
If $x \in X$ is such that every expected occurrence of $W_l$ is good, then $\phi (x) $ is a shift of $x$.
\end{proposition}

\begin{proof}
Suppose $x \in X$ is such that every expected occurrence of $W_l$ is good.  Choose $i\in \Z$ so that $x$ has an expected occurrence of $W_l$ at $i$.  We know that $\phi (x)$ has an expected occurrence of $W_l$ beginning at $i-r_{x,i}$.  Repeated applications of Lemmas \ref{lemmabasic1} and \ref{lemmabasic2} show that $\phi (x) = \sigma ^{r_{x,i}} (x)$.
\end{proof}

 \begin{definition}  For $n > l$, an expected occurrence of $W_{n}$ beginning at $i$ in $x \in X$ is  {\em totally good} if each expected occurrence of $W_l$ that is contained in the interval $[i, i + |W_n|)$ is good.
 
For $n > l$, an expected occurrence of $W_n$ beginning at $i$ in $x$ is  {\em totally bad} if each expected occurrence $W_l$ that is contained in the interval $[i, i + |W_n|)$ is bad.
 \end{definition}

If $x\in X$ has a totally good occurrence of $W_n$ beginning at $i$, then this occurrence of $W_n$ is expected. Moreover, repeated application of Lemma \ref{lemmabasic1} implies that $\phi(x)$ has an occurrence of $W_n$ beginning at $i-r_{x,i}$. In general, this occurrence need not be expected.

\begin{lemma}
\label{lemmaadvanced}
Let $n > l$ and suppose $x \in X$ has an occurrence of $W_{n} 1^s W_{n}$ beginning at $i$ and the $W_{n}$ beginning at $i$ in $x$ is totally good. Suppose further that $\phi(x)$ has an expected occurrence of $W_{n}$ beginning at $i - r_{x,i}$.  Let $t \in \N$ be such that $\phi (x)$ as an occurrence of $W_{n} 1^t W_{n}$ beginning at $i - r_{x,i}$.  
\begin{enumerate}
\item  If $s=t$, then the occurrence of $W_{n}$ beginning at $i + |W_{n}| + s$ is totally good.
\item  If $s \neq t$, then the occurrence of $W_{n}$ beginning at $i + |W_{n}| + s$ is totally bad.
\end{enumerate}
\end{lemma}

\begin{proof}
Since the occurrence of $W_n$ beginning at $i$ in $x$ is totally good, it is expected.  This implies that the occurrence of $W_n$ beginning at $i + |W_{n}| + s$ is expected.

If $s = t$, then repeated application of Lemma \ref{lemmabasic1} shows that the occurrence of $W_n$ beginning at $i + |W_{n}| + s$ is totally good.

If $s \neq t$.  Then a single application of Lemma \ref{lemmabasic1} shows that the expected occurrence of $W_l$ beginning at $i + |W_{n}| + s$ is bad.  Then repeated application of Lemma \ref{lemmabasic2} shows that each expected occurrence of $W_l$ that is contained in the interval $[i + |W_n| + s, i + 2|W_n| + s)$ is bad.  Thus, the occurrence of $W_{n}$ beginning at $i + |W_{n}| + s$ is totally bad.
\end{proof}

Lemma \ref{lemmaadvanced} immediately implies the following proposition.

\begin{proposition}
\label{proptotallybadW_n}
Let $n>l$ and suppose $x \in X$ has a totally good occurrence of $W_n$ beginning at $i$ and $\phi (x)$ has an expected occurrence of $W_n$ beginning at $i - r_{x,i}$.  If $j>i$ is as small as possible so that $x$ has a bad occurrence of $W_l$ beginning at $j$, then $x$ has a totally bad occurrence of $W_n$ beginning at $j$.
\end{proposition}

\begin{proposition}
\label{proptotallygoodW_{n+2}}
Let $n>l$ and suppose $x \in X$ has a totally good occurrence of $W_{n+2}$ beginning at $i$.  Then $x$ has a totally good occurrence of $W_n$ beginning at $i$ and $\phi(x)$ has an expected occurrence of $W_n$ beginning at $i - r_{x,i}$.
\end{proposition}

\begin{proof}
It is obvious that $x$ has a totally good occurrence of $W_n$ beginning at $i$. To prove the proposition we need only to show that $\phi(x)$ has an expected occurrence of $W_n$ beginning at $i - r_{x, i}$.  Since $x$ has a good occurrence of $W_l$ beginning at $i$, there is an occurrence of $W_l$ in $\phi (x)$ beginning at $i - r_{x, i}$.  Repeated application of Lemma \ref{lemmabasic1} implies that $\phi(x)$ as an occurrence of $W_{n+2}$ beginning at $i - r_{x,i}$.  Now, by Lemma \ref{propW_{n+2}}, $\phi(x)$ has an expected occurrence of $W_n$ beginning at $i - r_{x, i}$.
\end{proof}

\begin{proposition}
\label{propcontradiction}
There is no $x \in X$ such that for some $d >0$ there is a string of $d$ many bad occurrences of $W_l$ that is immediately preceded by a string of $2Q^3d$ many good occurrences of $W_l$ and is immediately followed by a good occurrence of $W_l$.
\end{proposition}

\begin{proof}
Suppose, towards a contradiction, that $d > 0$ and $x \in X$ has a string of $d$ many bad occurrences of $W_l$ beginning at $j$ that is immediately preceded by a string of $2Q^3d$ many good occurrences of $W_l$ beginning at $i$ and that is immediately followed by a good occurrence of $W_l$.

For $n \geq l$, let $N(l,n)$ denote the number of expected occurrences of $W_l$ in $W_n$.  Note that for $n \geq l$, $N(l, n+3) \leq Q^3 N(l, n)$.  

Choose $m > l $ so that $$ N(l, m-1) \leq d < N(l, m).$$

We will now show that $x$ has an expected occurrence of $W_{m+2}$ completely contained in the interval $[i, j)$.  We know that every expected occurrence of $W_l$ is contained in exactly one expected occurrence of $W_{m+2}$.  It follows that every string of at least $2N(l, m+2) - 1$ expected occurrences of $W_l$ contains exactly one expected occurrence of $W_{m+2}$.  But we also know that $N(l, m-1) \leq d$, which implies that $$2N(l, m+2) \leq 2 Q^3 N(l, m-1) \leq 2 Q^3 d.$$  Since there is a string of $2Q^3$ good occurrences of $W_l$ beginning at $i$ and ending with the last zero before $j$, There must be an expected occurrence of $W_{m+2}$ completely contained in the interval $[i, j)$.  

Now, it follows from Propositions  \ref{proptotallybadW_n} and \ref{proptotallygoodW_{n+2}} that $x$ has a totally bad occurrence of $W_m$ beginning at $j$.  Thus, there is a string of $N(l,m)$ bad occurrences of $W_l$ beginning at $j$.      
Since $d< N(l,m)$, this contradicts the fact that $x$ has a string of $d$ bad occurrences of $W_l$ which is followed by a string of one good occurrence of $W_l$.
\end{proof}

We will now prove that $\phi$ is an integral power of $\sigma$, thus completing the proof of Theorem \ref{thmtrivialcentralizer}.  We know from Propositions \ref{propmuae} and \ref{propcontradiction} that for $\mu$ almost every $x \in X$, every expected occurrence of $W_l$ is good.  Proposition \ref{propallgood} then implies that for almost every $x \in X$, there is some $i \in \Z$ so that $\phi (x) = \sigma ^i (x)$.  For each $i \in \Z$, let $A_i = \{x \in X : \phi (x) = \sigma^i (x)\}$.  Since $\phi$ commutes with $\sigma$, each $A_i$ is $\sigma$-invariant.  There must be some $i$ for which $\mu [A_i] > 0$ and by the ergodicity of $\sigma$, it must be that $\mu [A_i] = 1$.  Since for $\mu$ almost every $x \in X$, $\phi (x) = \sigma^i (x)$, we have that $\phi = \sigma^i$.  Since $\phi$ is an integral power of $\sigma$, so is $\psi$.  This completes the proof Theorem \ref{thmtrivialcentralizer}.

\section{Characterizing total ergodicity}

In this section we characterize total ergodicity for rank-one transformations with bounded cutting parameters. Recall that a measure preserving transformation $T$ is {\it totally ergodic} if each $T^k$ is ergodic for any positive integer $k$.

The objective of this section is to prove the following expanded version of Theorem \ref{thmintrototalergodicity}.  

\begin{theorem}
\label{thmtotalergodicity}
Let $(q_n)$ and $(a_{n,i})$ be the cutting and spacer parameters for a rank-one transformation $T$ and suppose the cutting parameter is bounded.  The following are equivalent.
\begin{enumerate}
\item [(i)]  $T$ is totally ergodic.
\item [(ii)] For all $d>1$ and $N \in \N$, there exists $k \in \N$ such that $\mu [T^k(B_N) \cap B_N]>0$ and  $d \ndiv k$.
\item [(iii)] For all $d>1$ and $N \in \N$, there exist $n \geq N$ and $0 < i < q_n$ so that $d \ndiv h_N + a_{n,i}$.
\item  [(iv)] For all $d>1$ and $N \in \N$, there exist $n \geq N$ and $0 \leq l < h_{n+1}$ such that $T^l (B_{n+1}) \subseteq B_n$ and $d \ndiv l$.
\end{enumerate}
\end{theorem}

It is relatively straightforward to show that (i) $\rightarrow$ (ii) $\rightarrow$ (iii) $\rightarrow$ (iv), even without the assumption that the cutting parameter is bounded.  The real work is in showing that (iv) $\rightarrow$ (i), and here we do need the bounded cutting parameter assumption.  

There is some preliminary work that we will do before giving the proof of this theorem.

We first remind the reader of the usual proof that every rank-one transformation $T$ is ergodic.  Suppose $A$ has positive measure and is $T$-invariant.  Our objective is to show that $A$ has full measure.  Let $\epsilon>0$ and choose $N_0 \in \N$ so that the measure of the stage-$N_0$ tower is greater than $1- \epsilon$.  Since the levels of the towers form a dense subset of the measure algebra, we can find some $N > N_0$ and some $0 \leq l < h_N$ so that $$\frac{\mu [T^l (B_N) \setminus A]}{\mu[ T^l (B_N)]} < \epsilon,$$ i.e., so that {\em most} of the points in the level $T^l (B_N)$ belong to $A$.  Since $A$ is $T$-invariant, the same must be true of every level of the stage-$N$ tower, i.e., for every $0 \leq l < h_N$, $$\frac{\mu [T^l (B_N) \setminus A]}{\mu[ T^l (B_N)]} < \epsilon.$$  This implies that the measure of the points in the stage-$N$ tower but not in $A$ is less than $\epsilon$.  Since $N > N_0$, the measure of the points not in the stage-$N$ tower is less than $\epsilon$.  Thus $\mu [A] > 1 - 2\epsilon$.  Since $\epsilon$ was arbitrary, $A$ has full measure.

Now we outline the argument that (iv) implies (i), in Theorem \ref{thmtotalergodicity}.  Assume (iv) holds, let $k>1$, and let $A$ be a set of positive measure that is $T^k$-invariant.  Our objective is to show that $A$ has full measure.  We let $\epsilon >0$ and choose $N_0 \in \N$ so that the measure of the stage-$N_0$ tower is greater than $1-\epsilon$.  For technical reasons (used in Lemma \ref{lemmatotalergodicity}), we also require that $h_{N_0} \geq k$.  We then choose $\delta \ll \epsilon$ and find $N>N_0$ and $0 \leq l_0 <  h_N$ so that $$\frac{\mu [T^{l_0} (B_N) \setminus A]}{\mu[ T^{l_0} (B_N)]} < \delta,$$ i.e., so that {\em most} of the points in the level $T^{l_0} (B_N)$ belong to $A$.  Then, since $A$ is $T^k$-invariant, we have that for all $0 \leq l < h_N$ with $l$ congruent to $l_0$ mod $k$, $$\frac{\mu [T^l (B_N) \setminus A]}{\mu[ T^l (B_N)]} < \delta.$$  We then (and herein lies the real work) find $M\geq N$ so that for {\em every} $0 \leq l < h_M$, $$\frac{\mu [T^l (B_M) \setminus A]}{\mu[ T^l (B_M)]} < \epsilon.$$  This implies that the measure of the points in the stage-$M$ tower but not in $A$ is less than $\epsilon$.  Since $M > N_0$, the measure of the points not in the stage-$M$ tower is less than $\epsilon$.  Thus $\mu [A] > 1 - 2\epsilon$.  Since $\epsilon$ was arbitrary, $A$ has full measure.

In order to fill in the outline above we need to specify how small $\delta$ must be and how to choose $M \geq N$ appropriately.  To do this we need Lemma \ref{lemmatotalergodicity} below, which in turn requires some notation and a few observations.  The notation, observations, and the lemma should be viewed in the context of proving (iv) implies (i) in the above outline.  In particular, we are assuming that:
\begin{itemize}
\item  $(q_n)$ and $(a_{n,i})$ are the cutting and spacer parameters for a rank-one transformation $T$, and the cutting parameter is bounded with bound $Q$;
\item  for all $d>1$ and $N \in \N$, there exist $n \geq N$ and $0 \leq l < h_{n+1}$ such that $T^l (B_{n+1}) \subseteq B_n$ and $d \ndiv l$; and
\item  $k>1$ and $A$ is a $T^k$-invariant set of positive measure.
\end{itemize}      

Here is the notation we will need.

\begin{itemize}
\item  If $l$ is a non-negative integer, let $[l]_k$ denote the congruence class of $l$ mod $k$, i.e., the unique $s \in \{0,1, \ldots, k-1\}$ so that $l$ and $s$ are congruent mod $k$.  Also, if $S \subseteq \N$ then $[S]_k = \{[s]_k: s \in S\}$.
\item  Let $Z$ be a measurable set and $\delta >0$.  We say $Z$ is {\em $\delta$-almost contained in $A$}, and write $Z \subseteq_\delta A$, iff $$\frac{\mu [Z \setminus A]}{\mu[Z]} < \delta.$$
\item  Let $N \in \N$, $\delta >0$ and $S \subseteq \{0, 1, \ldots, k-1\}$.  We say $S$ is {\em $(N, \delta)$-good} if for all $0 \leq l < h_N$ with $[l]_k \in S$, we have $T^l (B_N) \subseteq_\delta A$.  
\end{itemize}

Here are the observations.
\begin{enumerate}
\item  Suppose $Z$ is partitioned into the sets $Z_1, Z_2, \ldots, Z_n$.  
\begin{enumerate}
\item  If $Z_i \subseteq_\delta A$ for all $i$, then $Z \subseteq_\delta A$.
\item  If $Z \subseteq_\delta A$, then for some $i$, $Z_i \subseteq_\delta A$.
\item  If $Z \subseteq_\delta A$ and $\mu[Z_i] = \mu[Z_j]$ for all $i$ and $j$, then $Z_i \subseteq_{n\delta} A$ for all $i$.
\item  If $Z \subseteq_\delta A$ and $\mu[Z_i] \leq 2\mu[Z_j]$ for all $i$ and $j$, then $Z_i \subseteq_{2n\delta} A$ for all $i$.
\end{enumerate}
All of these facts are easy consequences of the definition of $\subseteq_{\delta}$.

\item  To show that $S \subseteq \{0,1, \ldots, k-1\}$ is $(N, \delta)$-good, it suffices to find, for each $s \in S$, a single $0 \leq l < h_N$ with $[l]_k = s$ so that $T^l (B_N) \subseteq_\delta A$.  This follows immediately from the fact that $A$ is $T^k$-invariant.

\end{enumerate}

And here is the lemma.

\begin{lemma} 
\label{lemmatotalergodicity}
Suppose $S \subseteq \{0,1, \ldots, k-1\}$ is $(N, \delta)$-good with $0 < |S| < k$ and $h_N \geq k$.  Then there are $N^\prime \geq N$ and $S^\prime \subseteq \{0,1, \ldots, k-1\}$ with $|S^\prime| > |S|$ such that $S^\prime$ is $(N^\prime, 2kQ\delta)$-good.
\end{lemma}

\begin{proof}
Consider $D = \{d \in \Z : [S+d ]_k = S\}.$  Clearly $D$ is a subgroup of $\Z$, and thus $D = d_0\Z$ for some $d_0 \in \N$.  Since $k \in D$, $d_0 \neq 0$.   Since $0 < |S| < k$, $1 \notin D$ and hence, $d_0\neq 1$.  

Now choose $n \geq N$ and $0 \leq l <h_{n+1}$ so that $T^l (B_{n+1}) \subseteq B_n$ and $d_0\ndiv l$.  Note that $l \notin D$ and hence, $[S+l]_k \neq S$.  

{\em Claim 1:}  There is some $i$ so that $[S+i]_k$ is $(n, 2k\delta)$-good.

{\em  Proof:}  
We know $S \subseteq \{0, 1, \ldots, k-1\}$ is $(N, \delta)$ good.  So if $0 \leq l < h_N$ and $[l]_k \in S$, then $T^l(B_N) \subseteq_\delta A$.  Let $L_S = \{0 \leq l < h_N : [l]_k \in S\}$ and let $$Z = \bigcup_{l \in L_S} T^l (B_N) .$$  By observation (1a), we have $Z \subseteq_\delta A$.  

We can also view $Z$ as the disjoint union of all levels of the stage $n$ tower of the form $T^{i + l} (B_n)$, where $0 \leq i < h_n$ is such that $T^{i} (B_n) \subseteq B_N$ and $l \in L_S$.  By observation (1b), there must be some $0 \leq i < h_n$ with $T^{i} (B_n) \subseteq B_N$ such that letting $$Y = \bigcup_{l \in L_S} T^{i + l}(B_n)$$ we have $Y \subseteq_\delta A$. 

For $s \in S$, let $L_s =  \{0 \leq l < h_N : [l]_k =s\}$.  Let $$Y_s = \bigcup_{l \in L_s} T^{i + l} (B_n)$$ and note that $\{Y_s : s \in S\}$ is a partition of $Y$.  Note that for each $s \in S$, the cardinality of $L_s$ is between the floor and the ceiling of $\frac{h_N}{k}$.  Since $h_N \geq k$, the floor of $\frac{h_N}{k}$ is at least 1.  Thus, $|L_s| \leq 2 |L_{s^\prime}|$ for all $s,s^\prime \in S$.  Since each level of the stage-$N$ tower has the same measure, $\mu[Y_s] \leq 2 \mu[Y_{s^\prime}]$ for all $s, s^\prime \in S$.  Thus, by observation (1d), we have $Y_s \subseteq_{2|S| \delta} A,$ for all $s \in S$.  

For $s \in S$, we know that $Y_s \subseteq_{2 |S| \delta} A$ and hence, by observation (1b), there is some $l \in L_s$ so that $T^{i + l} (B_n) \subseteq_{2|S| \delta} A$.  By observation (2), $[S+ i]_k$ is $(n, 2|S|\delta)$-good.    As $k>|S|$, we also know $[S+ i]_k$ is $(n, 2k\delta)$-good.   This finishes the proof Claim 1.

{\em  Claim 2:}  If $0 \leq j <h_{n+1}$ is such that $T^{j} (B_{n+1}) \subseteq B_n$, then $[S+i + j]_k$ is $(n+1, 2kQ\delta)$-good.   

{\em  Proof:}  Since $[S + i]_k $ is $(n, 2k\delta)$-good, we know that for $r \in [S+i]_k$, $$T^r(B_n) \subseteq_{2k\delta} A.$$  Note that $T^r (B_n)$ is the disjoint union of at most $Q$ levels of the stage-$(n+1)$ tower, each having the same measure and one of them being $T^{j +r} (B_{n+1})$.  This implies, by observation (1c), that $$T^{ r+j} (B_{n+1}) \subseteq_{Q2k\delta} A.$$  By observation (2), $[S+ i + j]_k$ is $(n+1, 2kQ\delta)$-good.  This finishes the proof of Claim 2.

Applying Claim 2 twice (independently), once with $j=0$ and once with $j = l$, shows that both $[S + i]_k$ and $[S + i + l]_k$ are $(n+1, 2kQ\delta)$-good.  Letting $N^\prime = n+1$ and $S^\prime = [S + i +l ]_k \cup [S + i]_k$, we have that $S^\prime$ is $(N^\prime, 2kQ\delta)$-good.  Since $[S+ l]_k \neq [S]$, we know that $[S + i + l]_k \neq [S + i]_k$ and hence, $|S^\prime| > |S|$.
\end{proof}

We now give the proof of Theorem~\ref{thmtotalergodicity}.

\begin{proof}[Proof of Theorem~\ref{thmtotalergodicity}]
We first show that (iv) implies (i).  Suppose (iv) holds.  Let $k>1$ and let $A$ be a set of positive measure that is $T^k$-invariant.  To show (i) holds, we need to show that $A$ has full measure.  Let $\epsilon >0$ and choose $N_0 \in \N$ so that the measure of the stage-$N_0$ tower is greater than $1-\epsilon$ and so that $h_{N_0} \geq k$.  Let $\delta = \epsilon/(2kQ)^{k-1}$ and find $N>N_0$ and $0 \leq l <  h_N$ so that $T^l (B_N) \subseteq_\delta A$.  Thus, by observation (2), we know $\{[l]_k\}$ is $(N, \delta)$-good.  By successively applying Lemma \ref{lemmatotalergodicity} as many as $k-1$ times, we can find $M \geq N$ such that $\{0, 1, \ldots, k-1\}$ is $(M, (2kQ)^{k-1} \delta)$ good.  Thus, since $(2kQ)^{k-1} \delta = \epsilon$, for every $0 \leq l < h_M$, $T^l (B_M) \subseteq_\epsilon A$.  This implies that the measure of the points in the stage-$M$ tower but not in $A$ is less than $\epsilon$.  Since $M > N_0$, the measure of the points not in the stage-$M$ tower is less than $\epsilon$.  Thus $\mu [A] > 1 - 2\epsilon$.  Since $\epsilon$ was arbitrary, $A$ has full measure.  This completes the proof that (iv) implies (i).

We now show that (i) implies (ii).  Suppose (ii) does not hold and fix $d>1$ and $N \in \N$ so that whenever $k\in \N$ is such that $\mu [T^k (B_N) \cap B_N] > 0$, $d | k$.   

We claim that $T^d$ is not ergodic, which implies that $T$ is not totally ergodic.  It suffices to find a positive measure set whose forward orbit does not have full measure.  Let $$A = \bigcup_{r \in \N} T^{rd} (B_N).$$  We need to show that $A$ does not have full measure.  Suppose, to the contrary, that $A$ does have full measure.  Then $\mu [A \cap T(B_N)] > 0$, and thus there is some $r \in \N$ so that $\mu[T^{rd} (B_N) \cap T(B_N)] > 0$.  This implies that $\mu [T^{rd - 1} (B_N) \cap B_N] > 0$ and thus, $d | (rd - 1)$, a contradiction.  Thus $A$ does   This completes the proof that (i) implies (ii).

We now show that (ii) implies (iii).  Suppose (ii) holds, but (iii) does not.  Since (iii) does not hold, we can find $d > 1$ and $N \in \N$ so that for all $n \geq N$ and $0<i < q_n$, $$d | h_N + a_{n,i}.$$  For $M \geq N$, we can view the stage-$M$ tower as a collection of $N$-blocks, with each pair of consecutive $N$-blocks separated by exactly $a_{n,i}$ spacers, for some $N \leq n < M$ and $0<i < q_n$.  Since we know $d | h_N + a_{n,i}$ for all $n \geq N$ and $0<i < q_n$ we have that if $T^l (B_M)$ is the base of one of the $N$-blocks, then $d | l$.

Now, since (ii) holds, we can find $k \in \N$ so that $\mu [B_N \cap T^k (B_N)] >0$ and $d \ndiv k$.   Let $\alpha = \mu [B_N \cap T^k (B_N)] $ and choose $M >N$ large enough that 
$$\mu \left[\bigcup_{l =k}^{h_M -1} T^l (B_M)\right] > 1-\alpha.$$  Therefore, there must be some $k \leq l < h_M$ so that $$B_N \cap T^k (B_N) \cap T^l (B_M) \neq \emptyset .$$  Since $B_N \cap T^l (B_M)$ is nonempty, we know that $T^l (B_M)$ is in the base of one of the $N$-blocks, and thus $d | l$.  Also, since $T^k (B_N) \cap T^l (B_M)$ is nonempty, we know $B_N \cap T^{l - k} (B_M)$ is nonempty.  This implies that $T^{l-k}(B_M)$ is the base of one of the $N$-blocks, which implies $d | (l-k)$.  Thus $d | k$, a contradiction.  This completes the proof that (ii) implies (iii).

We now show (iii) implies (iv).  Suppose (iii) holds and fix any $d >1$ and $N \in \N$.  We need to show there are $n \geq N$ and $0 \leq l < h_n$ such that $T^l (B_{n+1}) \subseteq B_n$ and $d \ndiv l$.  Since (iii) holds, we know there exist $n \geq N$ and $0<i< q_n$ such that $d$ does not divide $h_N + a_{n,i}$.  Fix such an $n$ and $i$ with $n\geq N$ as small as possible.

We first remark that the minimality of $n$ implies that $h_n$ is congruent to $h_N$ mod $d$.  Indeed, we can view the stage-$n$ tower as a collection of $N$-blocks with each pair of consecutive $N$-blocks separated by exactly $a_{n^\prime, i ^\prime}$ spacers for some $N \leq n^\prime < n$ and $0 < i^\prime < q_{n^\prime}$.  We know by the minimality of $n$ that for each such $a_{n^\prime, i^\prime}$, $$d | h_N +  a_{n^\prime, i^\prime}.$$ This implies that $h_n$ is congruent to $h_N$ mod $d$.  

Next we consider the stage-$(n+1)$ tower as a collection of $n$-blocks with each pair of consecutive $n$-blocks separated by $a_{n, j}$ for some $0<j<q_n$.  Define $l_j$ for $0 \leq j < q_n$ as follows:  $l_0 =0$ and $l_{j+1} = l_j + h_n + a_{n,j}$.  Notice that $\{l_j : 0 \leq j < q_n\}$ consists precisely of those $0 \leq l < h_{n+1}$ such that $T^l (B_{n+1}) \subseteq B_n$.  

We know that $d \ndiv (h_N + a_{n,i})$ and we also know that $h_n$ is congruent to $h_N$ mod $d$.  Thus $d \ndiv (h_n + a_{n,i})$.  Thus at least one of $l_i$ and $l_{i+1}$ is not divisible by $d$.  In either case there is some $0 \leq l < h_{n+1}$ such that $T^l (B_{n+1}) \subseteq B_n$ and $d \ndiv l$.  This completes the proof that (iii) implies (iv).
\end{proof}

\section{Concluding remarks}

Using the main result of Ryzhikov in \cite{Ryzhikov} and Theorems \ref{thmintrotrivialcentralizer2} and \ref{thmintrototalergodicity} from this paper, we now have a simple procedure for determining whether a bounded rank-one transformation has MSJ.  Suppose $(q_n)$ and $(a_{n,i})$
are bounded cutting and spacer parameters for a rank-one transformation $T$.  To determine whether $T$ has MSJ, check whether the following statements are both true.
\begin{enumerate}
\item  There exists $k \in \N$ such that for all $N \in \N$ there exist $n, m, i,j$, with $N \leq n,m < N+k$, $0<i<q_n$ and $0<j<q_m$, such that $a_{n,i} \neq a_{m,j}$.
\item  For all $N \in \N$ and $d>1$ there exist $n \geq N$ and $0 < i < q_n$ such that $d \ndiv h_N + a_{n,i}.$
\end{enumerate}
Theorem \ref{thmintrotrivialcentralizer2} tells us that $T$ has trivial centralizer iff the first statement is true.  Theorem $\ref{thmintrototalergodicity}$ tells us that $T$ is totally ergodic iff the second statement is true.  And Ryzhikov's theorem tells us that $T$ has trivial centralizer iff $T$ has trivial centralizer and is totally ergodic.

We remark that some parts of this procedure are valid if only the cutting parameter is assumed to be bounded.  In particular, bounded cutting parameter is enough for Theorem \ref{thmintrototalergodicity} to guarantee that $T$ is totally ergodic iff the second statement is true.  Also, bounded cutting parameter is enough to show that if the first statement is false, then $T$ does not have trivial centralizer (see Proposition \ref{propnottrivialcentralizer}).  Finally, if $T$ does not have trivial centralizer or is totally ergodic, then $T$ cannot have MSJ.  We do not know whether the other parts of this procedure are valid if only the cutting parameter is assumed to be bounded.

\end{document}